\documentclass[reqno]{amsart}
\usepackage{amssymb}
\usepackage{amsmath}
\usepackage{amsthm}
\usepackage{mathtools}
\usepackage{graphicx}
\usepackage{tikz}
\usetikzlibrary {shapes.geometric}
\usepackage{comment}
\newtheorem{thm}{Theorem}
\newtheorem{prop}[thm]{Proposition}
\newtheorem{lem}[thm]{Lemma}
\newtheorem{cor}[thm]{Corollary}
\theoremstyle{definition}

\theoremstyle{remark}

\newcommand{\vol}{\text{Vol}}
\newcommand{\BR}{\mathbb{R}}
\newcommand{\BC}{\mathbb{C}}

\newcommand{\Div}{\text{div}}
\newcommand{\Real}{\text{Re}}
\def\Xint#1{\mathchoice
{\XXint\displaystyle\textstyle{#1}}%
{\XXint\textstyle\scriptstyle{#1}}%
{\XXint\scriptstyle\scriptscriptstyle{#1}}%
{\XXint\scriptscriptstyle\scriptscriptstyle{#1}}%
\!\int}
\def\XXint#1#2#3{{\setbox0=\hbox{$#1{#2#3}{\int}$ }
\vcenter{\hbox{$#2#3$ }}\kern-.6\wd0}}

\def\dashint{\Xint-}

\begin{document}

\title{Results on Gradients of Harmonic Functions on Lipschitz Surfaces}

\author{Benjamin Foster}
\email{bfost@stanford.edu}

\begin{abstract}
We study various properties of the gradients of solutions to harmonic functions on Lipschitz surfaces. We improve an exponential bound of Naber and Valtorta \cite{NV} on the size of the superlevel sets for the frequency function to a sharp quadratic bound in this setting using complex analytic tools. We also develop a propagation of smallness for gradients of harmonic functions, settling an open question from \cite{LM} in this setting. Finally, we extend the estimate on superlevel sets of the frequency to more general divergence-form elliptic PDEs with bounded drift terms at the cost of a subpolynomial factor.
\end{abstract}

 \maketitle
 \section{Introduction}
Throughout, we will study the size of the gradient for harmonic functions on 2-dimensional domains with Lipschitz Riemannian metrics. We denote the Laplace-Beltrami operator by $\Delta_g$, and a function $u$ is said to be harmonic if it solves the equation $\Delta_g u=0$. In local coordinates, if we let $g_{ij}$ denote the components of the metric, $g^{ij}$ denote the components of the inverse tensor, and $|g|$ denote the determinant, the equation takes the form
\begin{equation}
    \frac{1}{\sqrt{|g|(x)}}\partial_i(\sqrt{|g|(x)}g^{ij}(x)\partial_ju(x))=0\qquad\text{in }B_2(0),
\end{equation}
where we have employed Einstein summation convention. We denote the maximum of the ellipticity and Lipschitz constants by $\lambda$, i.e.
\begin{equation}
    |g_{ij}(x)-g_{ij}(y)|\le \lambda |x-y|, \qquad \qquad \frac{1}{\lambda}|v|^2\le \sum_{i,j}g_{ij}v_iv_j\le \lambda|v|^2
\end{equation}
hold uniformly across all $x,y\in B_2(0)$ and all $v\in\BR^2$, for some $\lambda>1$.

Standard regularity results imply that the solution $u$ is $C^1$ \cite[Theorem 3.13]{HLBook}. Its critical set $\mathcal{C}(u)$, consisting of the points in $B_1(0)$ where $\nabla u$ vanishes, is finite. In fact, for nonconstant harmonic functions on the unit disk, the number of critical points can be bounded by a multiple of the \textbf{frequency function}\footnote{The frequency function will be defined more carefully in section \ref{sec: freq}.} of the solution \cite{HanNodal}. It was shown more recently in \cite{Zhu} that this holds more generally for solutions to divergence form elliptic equations with Lipschitz coefficients and a bounded drift term.

The frequency function of Almgren is a versatile tool that has been used in the study of elliptic equations\footnote{See for instance \cite{GL1}, \cite{GL2}, \cite{KZ}, and \cite{Tolsa} for applications to unique continuation problems, or \cite{CNV} and \cite{NV} for applications to the size of critical and singular sets of solutions to elliptic equations.}. It takes in a point in the domain and a scale parameter $r$ and satisfies a monotonicity property in the scale parameter, converging to the vanishing order of the solution at the point as $r$ tends to 0. When $r$ is instead fixed, it captures how the solution grows in size from the ball of radius $r$ to the ball of radius $2r$. Heuristically, the frequency function can be thought of as the analog of the ``degree of a polynomial" in the setting of solutions to elliptic equations; see \eqref{def freq sum} for a precise formulation. We direct the interested reader to \cite{PCMI} for a more complete exposition on the properties of the frequency function.

Quantitative bounds on the size of the critical set in terms of the frequency of the solution has been a topic of research interest, with the current best known bounds in higher dimensions established in \cite{NV}. It is known that the dimension of the critical set is at most $n-2$, and Lin conjectured in \cite{Lin} that 
\begin{equation}
    \mathcal{H}^{n-2}(\mathcal{C}(u)\cap B_{1/2}(0))\le CN^2,
\end{equation}
where $\mathcal{H}^{n-2}$ denotes the $(n-2)$-dimensional Hausdorff measure and $N$ is an upper bound for the frequency. The authors of \cite{NV} established Minkowski estimates on the volume of neighborhoods of the critical set by studying \textbf{effective critical sets} $\mathcal{C}_r(u)$, which have the advantage of behaving better under perturbation of the solution than the critical set\footnote{$\mathcal{C}_r(u)$ is defined more precisely in section \ref{sec: freq}}. The key fact was that these effective critical sets were covered by superlevel sets of the frequency function, which was the quantity whose size they were able to control. Their estimates are exponential in $N^2$ in general, although in dimension 2 they are only exponential in $N$. In dimension two, we can improve their bound to a quadratic bound that is sharp. This is our first main result.
\begin{thm}\label{main eff}
Suppose $u$ is a solution to $\Delta_g u=0$ on the Euclidean disk of radius 2. Assume its frequency satisfies $N_u(0,1)\le \Lambda$. Then we have the following volume estimate for the superlevel sets of the frequency function
\begin{equation}
    \vol(\{x:N(x,r)>c\}\cap B_{1/2}(0))\le C\Lambda^2 r^2,
\end{equation}
where $c<1$ is a suitably chosen constant and where $C>0$ is large enough, both depending on the Lipschitz and ellipticity constants of $g$.
\end{thm}
We also generalize this to more general divergence-form elliptic equations with a bounded drift term; see section \ref{section gen eqn}.

As mentioned before, a key property of the effective critical set at scale $r$ introduced by Naber and Valtorta is that it is covered by $\{x:N_u(x,r)>c_0\}$ for a suitable universal constant $c_0$ and contains a tubular neighborhood of the critical set. In higher dimensions, where sharp bounds on the size of the critical set have not been obtained, it can be helpful to informally think that the $(n-2)$-dimensional parts of the critical set occur when the solution is locally very symmetric and close to being a function of two variables. This can be made precise via stratification arguments; see \cite{Han} for a non-quantitative stratification and \cite{CNV}, \cite{NV} for a discussion of the quantitative stratification. The effective critical set is open, unlike the critical set which can change dimension when perturbed. Having control on the size of the superlevel sets of the frequency function (and hence also the effective critical set) in the two-dimensional case could be useful in understanding the size and structure of the critical set in higher dimensions as a consequence of the aforementioned quantitative stratifications.

It has also been a topic of interest in elliptic PDEs to understand how solutions are quantitatively ``close" to being constant or nonconstant in terms of the size of the set where the gradient is small, rather than the critical set. Results of interest can loosely be described by the estimate
\begin{equation}
    \sup_{B_{1/2}(0)}|\nabla u|\le C \sup_{E}|\nabla u|^{\alpha} \sup_{B_{1}(0)}|\nabla u|^{1-\alpha},
\end{equation}
where $E$ is often taken to be some sublevel set of the gradient; we call this propagation of smallness for the gradient. The question is what properties of the set $E$ allow us to deduce such an estimate and how the implicit constant $C$ depends on these properties. There are a number of well known results along these lines for both solutions to elliptic equations and their gradients, such as the Three Spheres Theorem. There has been interest in understanding how propagation of smallness works for arbitrarily wild sets. It was shown in \cite{LM} that for solutions to divergence form elliptic equations with Lipschitz coefficients smallness propagates for the solutions from arbitrarily wild sets with positive $(n-1+\delta)$-dimensional Hausdorff content. The argument involved a novel induction on scales as well as some geometric-combinatorial lemmas from \cite{Log1}. For gradients of solutions to elliptic equations, it is believed that there should be even better propagation of smallness properties; that is, smallness should propagate from arbitrarily wild sets with positive $(n-2+\delta)$-dimensional Hausdorff content for any $\delta>0$. The authors of \cite{LM} also gave a modified argument for gradients which proved propagation of smallness from sets with positive $(n-1-c_n)$-dimensional Hausdorff content for some small constant $c_n>0$, but were not able to achieve the conjectured optimal propagation result.

In the two-dimensional case, we first use the various tools developed in the intermediate section to give a short proof of propagation of smallness for gradients of harmonic functions in the Euclidean case. This has also been studied using tools from potential theory in the past, see \cite[Theorem 2.1]{Mal} for instance. We then extend this to the case of Lipschitz Riemannian metrics by appealing to the existence theory of isothermal coordinates. In such coordinates, the metric is conformal to the Euclidean metric, allowing us to reduce to the harmonic case. The key is that we can control the derivative of the quasiconformal coordinate change map pointwise in terms of the Lipschitz constant $\lambda$ of our original elliptic equation, which can be seen through an analysis of the construction of the quasiconformal coordinate change via the solution of the Beltrami equation. This allows us to do a local propagation of smallness for solutions in coordinate charts, and a simple chaining argument allows us to extend the propagation of smallness globally. Our second main result is the following.

\begin{thm}\label{main prop}
Let $g$ be a Lipschitz Riemannian metric, where $|g_{ij}(x)-g_{ij}(y)|\le \lambda$ and $\lambda^{-1} |v|^2\le \sum_{i,j} g_{ij}(x)v_iv_j\le \lambda |v|^2$ hold for some $\lambda>0$ and all $x,y\in B_1(0)$ and all $v\in\BR^2$. Suppose that $\Delta_g u=0$ holds for some $u:B_1(0)\rightarrow\BR$ with $\Vert\nabla u\Vert_{L^{\infty}(B_1(0))}=1$. Suppose also that $|\nabla u|\le \epsilon$ on a set $E_{\epsilon}$ which has $\delta$-dimensional Hausdorff content (for some $\delta>0$) equal to $\beta>0$. Then
\begin{equation}
    |\nabla u(x)|\le C \epsilon^{\gamma}
\end{equation}
holds for all $x\in B_{1/2}(0)$, where $C,\gamma>0$ depend only on $\lambda, \delta,\beta$. 
\end{thm}
Here, we have given a local version of the result, but applying this in sufficiently small coordinate charts for a Lipschitz surfaces and chaining extends the result to a global propagation of smallness, with constants depending only on the surface itself.

Throughout the paper, we use the dimension 2 hypothesis in several fundamental ways. The first is the well known fact that if $u$ is a harmonic function on a subset of the plane, then the function $F(x+iy)=u_x(x,y)-iu_y(x,y)$ is a holomorphic function with $|F|=|\nabla u|$. This allows for the use of powerful tools from complex analysis. We mainly use the fact that the zeros of $F$ can be factored out, so that $F(z)=P(z)g(z)$ where $P(z)$ is a polynomial and $g(z)$ is nonvanishing. We can then study the frequency of polynomials and nonvanishing functions separately to understand the frequency of gradients of arbitrary harmonic functions. The dimension 2 hypothesis is also used in order to change to isothermal coordinates, as these do not exist in general in higher dimensions.

The structure of the paper is as follows. In section 2, we discuss the conventions for the frequency of a solution, giving slight variants on the normal definitions for technical reasons. In section 3, we study the frequency of solutions with nonvanishing gradients. In section 4, we study the level sets of the frequency of polynomials. In section 5, we use the previous sections to derive a quadratic bound (in terms of frequency) on the size of the superlevel sets of the frequency of a harmonic function in the Euclidean case. In section 6, we give a simple proof of the propagation of smallness result for gradients of harmonic functions in the Euclidean case. In section 7, we discuss how to extend the propagation of smallness for gradients and volume estimates for superlevel sets of the frequency to the more general case of harmonic functions on a Lipschitz surface. In section 8, we further extend this estimate to more general elliptic PDEs in divergence form with Lipschitz coefficients and a bounded measurable drift term. This relies on the fact that the gradient of a solution to such an elliptic equation also solves a divergence form elliptic equation, and the matter can be simplified using an isothermal change of coordinates. 

There is another recent paper by Zhu \cite{zhu2023remarks} which obtains a similar result to Theorem 2 of this paper for solutions to divergence-form elliptic equations without drift terms and appeared online shortly before the first version of this paper did. The methods of that paper helped inspire some of the arguments in section \ref{section gen eqn} in the second version of this paper.

The author is grateful to Eugenia Malinnikova for numerous valuable discussions while working on this project. The author completed part of this work while visiting the Hausdorff Research Institute for Mathematics, funded by the Deutsche Forschungsgemeinschaft (DFG, German Research Foundation) under Germany's Excellence Strategy – EXC-2047/1 – 390685813. The author is grateful for the institute's hospitality and support. The author also appreciates some funding from the NSF grant DMS-1956294.
 
\section{Conventions and Preliminaries}\label{sec: freq}
Throughout the paper, we will use $C$ to denote various constants. Unless otherwise specified, these will be universal in the harmonic setting, and these will depend only on the ellipticity and Lipschitz constants in the elliptic setting. We will also use the notation $A\lesssim B$ to denote $A\le CB$ for such constants $C$ when it is convenient. 

The key result will be a volume estimate on superlevel sets of the frequency. For technical reasons, we have slightly modified some of the definitions given in \cite{NV}. Given a nonconstant harmonic function $u$, we denote its frequency\footnote{More generally, if $u$ solves a divergence-form equation, then we can also use this formulation after making a linear change of coordinates so that the second order term is the Euclidean Laplacian at the point $x$} by
\begin{equation}
    N(x,r)=\log_2\frac{\dashint_{B_{2r}(x)}|\nabla u(y)|^2\,dy}{\dashint_{B_{r}(x)}|\nabla u(y)|^2\,dy}.
\end{equation}
If $u-u(x)=\sum_{d\ge 1} a_dP_d$ is the expansion into homogeneous harmonic polynomials about $x$, then the frequency can equivalently be expressed as
\begin{equation} \label{def freq sum}
    N(x,r)=\log_2 \frac{\sum_{d\ge 1} 2^{2d-2}da_d^2 r^{2d}}{\sum_{d\ge 1} da_d^2 r^{2d}}.
\end{equation}
Notice that the formulation \eqref{def freq sum} implies that frequency is monotone in $r$ for fixed $x$, just as is the case with the standard definition of frequency.
\begin{prop}\label{monotone}
$N(x,r)$ is monotone nondecreasing in $r$.
\end{prop}
\begin{proof}
For convenience, we replace $r$ by $\sqrt{r}$. Differentiating in $r$ gives
\begin{equation}
    \log(2)\partial_r N(x,\sqrt{r})= \frac{\sum_{d\ge 1}4^{d-1}d^2a_d^2r^{d-1}}{\sum_{d\ge 1}4^{d-1}da_d^2r^{d}}-\frac{\sum_{d\ge 1}d^2a_d^2r^{d-1}}{\sum_{d\ge 1}da_d^2r^{d}}.
\end{equation}
Then monotonicity is equivalent to the inequality
\begin{equation}\label{monot eq}
    \left(\sum_{d\ge 1}4^{d-1}d^2a_d^2r^{d-1}\right)\left(\sum_{d\ge 1}da_d^2r^{d-1}\right) \ge \left(\sum_{d\ge 1}4^{d-1}da_d^2r^{d-1}\right)\left(\sum_{d\ge 1}d^2a_d^2r^{d-1}\right).
\end{equation}
We can then divide everything by $\left(\sum_{d\ge 1}da_d^2r^{d-1}\right)$ to normalize, so without loss of generality $\left(\sum_{d\ge 1}da_d^2r^{d-1}\right)=1$ and we have a probability measure on the positive integers via $p(d)=da_d^2r^{d-1}$. Recall Chebychev's inequality for sums, which says that for nondecreasing functions $f,g$ and probability measure $p$, we have
\begin{equation}
    \sum_{d} f(d)g(d)p(d)\ge \left(\sum_d f(d)p(d)\right) \left(\sum_d g(d) p(d)\right).
\end{equation}
Taking $f(d)=4^{d-1}$ and $g(d)=d$ then implies \eqref{monot eq}, completing the proof. 
\end{proof}

We have taken the logarithm with base 2 in the definition so that if $u-u(x)$ has vanishing order $j$ at a point $x$, then $N(x,r)\ge 2j-2$ for all $r$. In particular, we have that $x$ is in the critical set if and only if $\lim_{r\rightarrow 0}N(x,r)\ge 2$. 

In the following discussion, we will let $u$ be a solution to a more general elliptic equation $\Div(A\nabla u)=0$. In \cite{CNV}, the effective critical set at scale $r$ was (roughly) defined to be the set of points such that the solution $u$ was well approximated by a normalized linear function on the ball of radius $r$. Subsequently, in \cite{NV}, the effective critical set at scale $r$ was defined to be
\begin{equation}\label{def: eff crit set}
    \tilde{\mathcal{C}}_r(u)=\left\{x:r^2\inf_{y\in B_r(x)}|\nabla u(y)|^2< C \dashint_{\partial B_{2r}(x)}|u(y)-u(x)|^2\,dy\right\}.
\end{equation}
A short calculation shows that all points in the original definition of the effective critical set are also in $\tilde{\mathcal{C}}_r(u)$. As is discussed in \cite{NV}, for a suitable choice of the constant $C$, this is contained in the superlevel set
\begin{equation}
    \tilde{\mathcal{C}}_r(u)\subset \{x:N(x,r)>C\}.
\end{equation}
Also notice that $\tilde{\mathcal{C}_r}(u)$ trivially contains an $r$-tubular neighborhood of the critical set. For this reason, we are interested in getting Minkowski estimates for superlevel sets of the frequency.

We will also make use of the local equivalence of $L^2$ and $L^{\infty}$ norms for harmonic functions, that is
\begin{equation}\label{loc equiv}
    C_1\Vert u\Vert_{L^2(B_{1/2}(0))}\le \Vert u\Vert_{L^{\infty}(B_{1/2}(0))} \le C_2\Vert u\Vert_{L^2(B_{1}(0))}.
\end{equation}
In particular, we can apply this to the partial derivatives of harmonic functions to get analogous equivalences for gradients of harmonic functions.

We also recall the definition of Hausdorff content of a set. Unlike Hausdorff measure, Hausdorff content is always finite for bounded sets which makes it more suitable for propagating smallness quantitatively from open sets (such as sublevel sets of the gradient). Given $\delta>0$, the $\delta$-dimensional Hausdorff content of a set $E$ is defined to be
\begin{equation}
    \mathcal{C}_H^{\delta}(E)=\inf \sum_j r_j^\delta
\end{equation}
where the infimum is taken over all covers of $E$ by balls $B_{r_j}(x_j)$. Unlike the Hausdorff measure, we do not require the diameters of the covering sets to become arbitrarily small. 

Finally, we recall the definition of the weak $L^p$ spaces. For $0<p<\infty$ and a measure space $(X,\mu)$, we define the quasinorm
\begin{equation}\label{weak lp def}
\Vert f\Vert_{L^{p,\infty}(X,\mu)} = \sup_{\gamma>0}\{\gamma\mu(\{x:|f(x)|>\gamma\})^{1/p}\}.
\end{equation}
and the space $L^{p,\infty}$ consists of all measurable functions $f$ such that $\Vert f\Vert_{L^{p,\infty}(X,\mu)}$ is finite. It is known that when $p>1$ and $(X,\mu)$ is $\sigma$-finite, there is a norm $|||\cdot|||_{L^{p,\infty}(X,\mu)}$ on $L^{p,\infty}$ which is equivalent to to \eqref{weak lp def}, see \cite[pp. 14-15]{Grafakos}.

\section{Empty Superlevel Sets for Nonvanishing Gradients}
The first thing we want to show is that if the critical set is empty, then the frequency is at most $1/2$ at small scales. Before we make this precise, we prove the following lemma:
\begin{lem} \label{lem: grad lower bd}
Let $u$ be a harmonic function on $B_2(0)$ with $\Vert \nabla u\Vert_{L^{\infty}(B_1(0))}=1$. Suppose its frequency satisfies $N(0,1)\le\Lambda$ and assume $\Lambda$ is sufficiently large. If $|\nabla u|$ is nonvanishing on the unit ball, then
\begin{equation}
    \big|\log|\nabla u(x)|\big|\leq C\Lambda
\end{equation}
for all $x\in B_{1/2}(0)$, where $C$ is universal.
\end{lem}
\begin{proof}
By monotonicity, we have that
\begin{equation}
    \log_2\left(\frac{\dashint_{B_{2}(0)}|\nabla u(y)|^2\,dy}{\dashint_{B_{1/2}(0)}|\nabla u(y)|^2\,dy}\right)=N(0,1)+N(0,1/2)\le 2\Lambda
\end{equation}
Since $u$ is harmonic, so is each partial derivative of $u$, so we have local equivalence of $L^2$ and $L^{\infty}$ norms up to doubling of the ball by \eqref{loc equiv}. In particular, this implies
\begin{equation}
    \dashint_{B_2(0)}|\nabla u|^2\ge c\Vert \nabla u\Vert_{L^{\infty}(B_1(0)}^2\ge c
\end{equation}
for some universal constant. We also can trivially upper bound the average over the half ball by the supremum. Together, these inequalities imply
\begin{equation}
    2\Lambda \ge \log(c) -\log\left(\sup_{B_{1/2}(0)}|\nabla u|\right)
\end{equation}
Rearranging implies
\begin{equation}
    \inf_{B_{1/2}(0)} \left(-\log|\nabla u(x)|\right) \le -\log(c)+2\Lambda
\end{equation}
Now, since $\nabla u$ is nonvanishing, recall that $-\log|\nabla u|$ is harmonic and by our normalization, it is nonnegative on the unit ball. Thus, applying Harnack's inequality on the half ball implies that the 
\begin{equation}
\sup_{B_{1/2}(0)}\lesssim \Lambda +C,
\end{equation}
giving the claim since $|\log|\nabla u(x)||=-\log|\nabla u(x)|$ and $\Lambda$ is large enough to absorb the additive constant.
\end{proof}
With this preliminary result established, it is not too difficult to show that if the critical set is empty, then so are large enough superlevel sets of the frequency, at least in the case where the gradient is nonvanishing.
\begin{prop} \label{prop: small scale freq bd}
Let $u$ be a harmonic function on the double of the unit ball with $|\nabla u(x)|>0$ everywhere on the unit ball. Suppose that the frequency bound $N(x,1)\le \Lambda$ holds across the unit ball for some sufficiently large $\Lambda>2$. Then there is a universal constant $c>0$ such that
\begin{equation}
N(x,c/\Lambda)\le \frac{1}{2}
\end{equation} for all $x$ in the unit ball. 
\end{prop}
\begin{proof}
Throughout, denote $v=-2\log|\nabla u|$ which is a nonnegative bounded harmonic function by hypothesis. We have that
\begin{equation}
    N(x,c/\Lambda)=\log_2\left(\frac{\dashint_{B_{2c/\Lambda}(x)} \exp(-v(y))\,dy}{\dashint_{B_{c/\Lambda}(x)}\exp(-v(y))\,dy}\right)=v(y_1)-v(y_2)
\end{equation}
for some points $y_j$ in $B_{jc/\Lambda}(x)$. We recall the following quantitative version of the Harnack inequality: if $w$ is positive and harmonic on the unit ball then
\begin{equation}
    \frac{1-|y|}{1+|y|}w(0)\le w(y) \le \frac{1+|y|}{1-|y|}w(0).
\end{equation}
Applying this at the points $y_1,y_2$ for $v$ and using that they are very close to $x$, we get that there is some universal constant $K$ such that the difference is controlled, i.e.
\begin{equation}
    |v(y_2)-v(y_1)|\le Kcv(x)/\Lambda
\end{equation}
However, Lemma \ref{lem: grad lower bd} showed that $v(x)\le C\Lambda$. Taking the original factor $c$ small enough, this implies that the frequency is at most $1/2$, as desired. 
\end{proof}
Note that for the volume estimates we want, we will always be able assume $\Lambda$ is large enough to satisfy the hypotheses in the results of this section.

\section{Holomorphic Polynomial Frequency Control}
As we are working in the two-dimensional case, we have the tools of complex analysis available. Given a harmonic function $u$, we can form a holomorphic function $F(x+iy)=u_x(x,y)-iu_y(x,y)$ with the property that $|F|=|\nabla u|$ everywhere. One useful tool we can utilize is that we can factor out roots of $F$. Since $F$ is a holomorphic function on the disk of radius 2, we can write
\begin{equation}
    F(z)=g(z)\prod_{j}(z-a_j)
\end{equation}
where $g(z)$ is nonvanishing in the unit disk and the $a_j$s are the roots of $F$ in the unit disk counted with multiplicity. We know the product is finite by analyticity. In the previous section, we already investigated the nonvanishing case, so we will now study the polynomial case.
\begin{prop}\label{prop: poly vol}
Let $P(z)=\prod_{j=1}^{\Lambda}(z-z_j)$ be a holomorphic polynomial on the disk, arising from a harmonic polynomial via $P\sim u_x-iu_y$. There is a universal constant $C$ such that the set where the frequency is greater than $1/2$ satisfies the following volume estimate:
\begin{equation}
    \vol(\{x\in B_{1/2}(0):N(x,r)>1/2\})\le C\Lambda^2 r^2.
\end{equation}
\end{prop}

\begin{proof}
Note that we may assume that $r<C\Lambda^{-1}$, as otherwise the bound is trivial due to the ball having finite measure. The frequency of $u$ is controlled by $\Lambda$, as can be seen by \eqref{def freq sum}. Fix a point $z_0$ in the unit disk which is distance at least $4r$ from each $z_j$. Then we have that for any $w\in B_{2 r}(z_0)$ and $z\in B_{r}(z_0)$
\begin{equation}
    \frac{|w-z_j|}{|z-z_j|}\leq 1+\frac{|w-z|}{|z-z_j|}\le 1+\frac{3r}{|z-z_j|}
\end{equation}
Optimizing in $w$ and $z$ leads to the estimate
\begin{align}
    N(z_0,r)&\leq \log_2 \frac{\sup_{w\in B_{2r}(z_0)} |P(w)|^2}{\inf_{z\in B_{r}(z_0)} |P(z)|^2} \\
    &\leq \sup_{z\in B_{r}(z_0)} 2\sum_{j=1}^{\Lambda}\log_2\left(1+\frac{3r}{|z-z_j|}\right) \\
    &\le \sup_{z\in B_{r}(z_0)} C\sum_{j=1}^{\Lambda}\frac{r}{|z-z_j|}
\end{align}
Let $B(r)$ be the union of balls of radius $4r$ centered at each $z_j$, so that $B(r)$ has volume at most $C\Lambda r^2$. Let $S(r)=B_1(0)\setminus B(r)$. We have by definition of the weak $L^2$ quasinorm that
\begin{equation}\label{eq: weak bound}
    \vol\left(\left\{z_0\in S(r):\sup_{z\in B_{r}(z_0)} \sum_{j=1}^{\Lambda}\frac{Cr}{|z-z_j|}\ge 1/2 \right\}\right)\le 4\left\Vert \sup_{z\in B_{r}(z_0)}\sum_{j=1}^{\Lambda}\frac{Cr}{|z-z_j|} \right\Vert_{L^{2,\infty}(S(r))}^2.
\end{equation}
 For a single term in the sum, we claim we have the estimate
\begin{equation}\label{weak l2}
\left\Vert \sup_{z\in B_{r}(z_0)}\frac{Cr}{|z-z_j|} \right\Vert_{L^{2,\infty}(S(r))}^2\le Cr^2.
\end{equation}
This can be seen by homogeneity and as a result of the well known fact that for any $w\in\BR^2$
\begin{equation}
    \left\Vert \frac{1}{|z-w|} \right\Vert_{L^{2,\infty}(\BR^2)}^2\le C,
\end{equation}
combined with the simple estimate
\begin{equation}
    \sup_{z\in B_{r}(z_0)}\frac{Cr}{|z-z_j|}\le \frac{C'r}{|z_0-z_j|},
\end{equation}
which is valid for all $z_0\in S(r)$. Thus, we deduce \eqref{weak l2}.

Since $L^{2,\infty}$ is normable, at the cost of a single multiplicative factor, we can apply the triangle inequality with arbitrarily many terms. In particular, this gives us the desired bound 
\begin{equation}\label{eq: final bound}
    \left\Vert \sup_{z\in B_{r}(z_0)}C\sum_{j=1}^{\Lambda}\frac{r}{|z-z_j|} \right\Vert_{L^{2,\infty}(S(r))}^2\le C\Lambda^2r^2.
\end{equation}
We have established that the frequency's superlevel set is contained in the union of balls $B(r)$, which has measure at most $C\Lambda r^2$, and the subset of $S(r)$ arising from the weak $L^2$ estimate which has measure at most $C\Lambda^2 r^2$, which gives the claim.
\end{proof}

This result can be seen to be sharp by considering monomials $z^{\Lambda}$ where $\Lambda$ is a positive integer. In fact, in the proof we showed the following slightly stronger result.

\begin{cor}\label{cor: poly vol}
Let $P(z)=\prod_{j=1}^{\Lambda}(z-z_j)$ be a holomorphic polynomial on the disk, arising from a harmonic polynomial via $P\sim u_x-iu_y$. There is a universal constant $C$ such that
\begin{equation}
    \vol\left(\{z_0\in B_{1/2}(0):\log_2 \frac{\sup_{w\in B_{2r}(z_0)} |P(w)|^2}{\inf_{z\in B_{r}(z_0)} |P(z)|^2}>1/2\}\right)\le C\Lambda^2 r^2.
\end{equation}
\end{cor}

\section{More General Holomorphic Functions}
Finally, we must deal with the general case where we have a factorization $F(z)=g(z)P(z)$ where $g$ is nonvanishing in the unit ball and $P$ is a polynomial with roots in the unit ball. It will be convenient to assume that $F$ is defined on a sufficiently large ball centered at the origin. We will work with analogs of the doubling index for these; given a holomorphic function $F$, define
\begin{equation}
    N_F(z,r)=\log_2\frac{\dashint_{B_{2r}(0)} |F(z)|^2\,dz}{\dashint_{B_{r}(0)}|F(z)|^2\,dz}.
\end{equation}
We first recall the fact that in two dimensions, controlling the frequency on a large disk controls the number of zeroes on a smaller disk.
\begin{prop}[Theorem 3.4 in \cite{HanNodal}]\label{root bd}
Suppose $N_F(0,1)\le \Lambda$. Then there exists some universal $s>0$ such that $F$ has at most $2\Lambda$ roots on $B_s(0)$, counting multiplicities. 
\end{prop}
As a result of this, we get a bound on the degree of $P$ in terms of $\Lambda$. We first claim a quasi-subadditivity property for the frequency given such a factorization $F=Pg$, where $P$ is monic.
\begin{lem}\label{lem: freq bd for nonzero factor}
Suppose $N_F(0,2t)\le \Lambda$ where $t=\max\{1/s,5\}$ with $s$ as in Proposition \ref{root bd}, $\Lambda$ is sufficiently large, and $F=Pg$ is a factorization into a monic polynomial whose roots all lie in $B_1(0)$ and a holomorphic function which is nonvanishing on $B_1(0)$. Then there exists some universal $C>0$ such that $N_g(0,t)\le C\Lambda$.
\end{lem}
\begin{proof}
We have by the previous lemma that the degree of $P$ is at most $2\Lambda$. Outside of the ball of radius $2<t/2$, we have that $1<|z-a_j|$ for any root $a_j$ of $P$ and inside the ball of radius $k$, we have $|z-a_j|\le k+1$. We estimate
\begin{align}
    2^{3\Lambda}&\ge\frac{\dashint_{B_{4t}(0)} |P(z)|^2|g(z)|^2\,dz}{\dashint_{B_{t/2}(0)}|P(z)|^2|g(z)|^2\,dz} \\
    &\ge c \frac{\int_{B_{4t}(0)-B_{2t}(0)} |P(z)|^2|g(z)|^2\,dz}{\int_{B_{t}(0)-B_{t/2}(0)}|P(z)|^2|g(z)|^2\,dz} \\
    &\ge \frac{c}{(t+1)^{4\Lambda}} \frac{\int_{B_{4t}(0)-B_{2t}(0)} |g(z)|^2\,dz}{\int_{B_{t}(0)-B_{t/2}(0)}|g(z)|^2\,dz} \\
    &\ge \frac{c}{(t+1)^{4\Lambda}} \frac{\int_{B_{2t}(0)} |g(z)|^2\,dz}{\int_{B_{t}(0)}|g(z)|^2\,dz} \\
    &\ge \frac{c}{(t+1)^{4\Lambda}} 2^{N_g(0,t)}
\end{align}
Here, we used subharmonicity to apply the mean value inequality, allowing us to control integrals over the half ball by integrals over a larger annulus. We also used the fact that outside the ball of radius 2, we have that $|P(z)|\ge 1$ and inside the ball of radius $k$, we have $|P(z)|\le (k+1)^{2\Lambda}$. The claim follows by rearranging and taking $\Lambda$ sufficiently large compared to $c$.
\end{proof}

We now have tools for understanding the superlevel sets of frequency for polynomials, nonvanishing functions, and how to combine these estimates for more general functions. We are ready to prove Theorem \ref{main eff} in the harmonic case.

\begin{prop} \label{main eff harm}
Suppose $u$ is a harmonic function on the disk of radius $4s$ for some suitably large universal constant $s$. Assume its frequency satisfies $N_u(0,2s)\le \Lambda$. Then we have the following volume estimate for the frequency superlevel sets
\begin{equation} 
    \vol(\{x:N(x,r)>1\}\cap B_{1/2}(0))\le C\Lambda^2 r^2
\end{equation}
for some universal constant $C$.
\end{prop}
\begin{proof}
As usual, define the holomorphic function $F(x+iy)=u_x(x,y)-iu_y(x,y)$. Based on our previous remarks, $N_F(0,2s)\le \Lambda$ and it suffices to show the desired volume estimate for the set of points where $N_F(z,r)>1$. Moreover, we may assume $r<c\Lambda^{-1}$ as otherwise, the volume estimate holds just by taking $C$ large enough since $B_{1/2}(0)$ is a set of finite measure. Factor $F(z)=P(z)g(z)$ where $P$ is a polynomial with all its roots in the unit disk and degree at most $2\Lambda$ and $g$ is nonvanishing in the unit disk. We have that $N_P(0,1)\le C\Lambda$, and it follows by Lemma \ref{lem: freq bd for nonzero factor} that $N_g(0,1)\le C\Lambda$. For simplicity, let $\mathcal{C}(u)$ denote only the critical set of $u$ inside the unit disk. Proposition \ref{prop: small scale freq bd} implies that if we choose $c$ small enough so that $r<c\Lambda^{-1}$ is sufficiently small, then $N_g(z,r)<1/2$ for any $z$ in the disk of radius $1/2$. Now, we will reuse the estimate on $P$ from the polynomial case. We have for $z$ in $B_{1/2}(0)$ that
\begin{align}
    2^{N_F(z,r)}&=\frac{\dashint_{B_{2r}(z)} |P(w)|^2|g(w)|^2\,dw}{\dashint_{B_{r}(z)}|P(w)|^2|g(w)|^2\,dw} \\
    &\le \frac{\dashint_{B_{2r}(z)} |g(w)|^2\,dw\sup_{B_{2r}(z)}|P(w)|^2}{\dashint_{B_{r}(z)}|g(w)|^2\,dw\inf_{B_r(z)}|P(w)|^2} \\
    &\le 2^{N_g(z,r)}\frac{\sup_{B_{2r}(z)}|P(w)|^2}{\inf_{B_r(z)}|P(w)|^2}.
\end{align}
Taking logarithms, the frequency from $g$ contributes at most $1/2$, and by Corollary \ref{cor: poly vol}, we see that we can bound the contribution from the polynomial to the frequency to be at most $1/2$ on everywhere except on a set of measure $C\Lambda^2r^2$. In particular, this implies that we have
\begin{equation}
    N_F(z,r)\le 1
\end{equation}
holds for all $z$ in $B_{1/2}(0)$ outside of a set of measure $C\Lambda^2r^2$, which gives the desired result.
\end{proof}
Moreover, this estimate is sharp in $\Lambda$. This can be seen by considering the holomorphic function $F(z)=\exp(\Lambda z)$, which arises as the gradient of the harmonic function $\Lambda^{-1}\exp(\Lambda x)\cos(\Lambda y)$. It can be seen to have frequency comparable to $\Lambda$ at unit scale, and its frequency is the full unit ball when $r=C/\Lambda$ for a suitable choice of constant. In fact, its frequency function is independent of the point chosen and depends only on the scale $r$; an easy estimate can lead to the lower bound $N(z,r)\ge \Lambda r/2-C$ for any $z$, so as $r\rightarrow 0$, there is a transition between the superlevel set $\{z: N(z,r)\ge 1\}$ being the full ball and being empty below $r=C'/\Lambda$. Thus, comparing this with \eqref{def: eff crit set}, we see the power of $\Lambda$ cannot be improved in the volume estimate.

\section{Propagation of Smallness for Gradients of Harmonic Functions}
With the tools developed in previous sections, we can prove the desired propagation of smallness for gradients of harmonic functions. Given a function $f$, we will use the following notation for its sublevel set
\begin{equation}
    E_a(f)=\{z:|f(z)|<e^{-a}\}
\end{equation}
We will need the following result of Cartan concerning polynomials.

\begin{lem}[\cite{Cartan}]\label{cover poly}
Let $P$ be a monic polynomial of degree $n$. Then for any $\delta,a>0$, we have that there is a finite collections of balls $\{B_j\}$ with radii $r_j$ covering $E_a(P)\cap B_1(0)$ and
\begin{equation}
    \sum_j r_j^{\delta}\le Ce^{-a\delta/n}
\end{equation}
\end{lem}
In particular, we have the following immediate corollary for a polynomial $P$ of degree $n$
\begin{equation}
    \mathcal{C}_{\mathcal{H}}^{\delta}\left(E_a(P)\cap B_1(0)\right)\le Ce^{-a\delta/n}.
\end{equation}
From here, obtaining the propagation of smallness is fairly straightforward.

\begin{thm}\label{thm: harm prop}
Let $u$ be a harmonic function on the double of the unit disk, and assume its gradient has $\Vert \nabla u\Vert_{L^{\infty}(B_1(0))}=1$. Suppose that
\begin{equation}
    \beta\le \mathcal{C}_{\mathcal{H}}^{\delta}(E_a(|\nabla u|)\cap B_1(0))
\end{equation}
for some $\delta,a>0$. Then we have that
\begin{equation}
    |\nabla u(x)|\le Ce^{-\gamma a}
\end{equation}
for all $x$ in the disk of radius $1/2$ and some $\gamma,C>0$ depending only on $\beta,\delta$.
\end{thm}
\begin{proof}
As usual, consider the holomorphic function $F(z)=u_x(x+iy)-iu_y(x+iy)$. Let $N_F(0,1)=\Lambda$. We claim that $\Lambda\gtrsim a$; then the conclusion of the theorem follows immediately by the normalization of $F$ in the hypotheses, the definition of frequency, and the local equivalence of $L^2$ and $L^\infty$ norms. Take the usual factorization $F(z)=P(z)g(z)$ where $P$ is a monic polynomial whose roots lie in the unit disk and $g$ is nonvanishing on the unit disk. We have the simple containment
\begin{align*}
    \{z\in B_1(0):|F(z)|<e^{-a}\}\subset &
    \,\,\{z\in B_1(0):|P(z)|<e^{-a/2}\}\\
    & \qquad\cup\{z\in B_1(0):|g(z)|<e^{-a/2}\}.
\end{align*}
Taking the Hausdorff content and applying Lemma \ref{cover poly}, we have that
\begin{equation} \label{key prop eq}
    \beta\le Ce^{-\delta a/n} +\mathcal{C}_{\mathcal{H}}^{\delta}(\{z\in B_1(0):|g(z)|<e^{-a/2}\})
\end{equation}
where $n\le C\Lambda$ is the degree of $P$. If $n\gtrsim a$, then we are done, as then $\Lambda \gtrsim a$ as well. If not, then the first term is at most $\beta/2$ (for fixed $\beta>0$ where the implied constant is chosen suitably), and we instead consider the frequency of $g$, call it $\rho$, which is at most $C\Lambda$. We have by Lemma \ref{lem: grad lower bd} that
\begin{equation}
    |g(z)|\ge e^{-c\rho}.
\end{equation}
In order for \eqref{key prop eq} to hold, we must have $\rho\gtrsim a$, which in turn implies by Lemma \ref{lem: freq bd for nonzero factor} that $\Lambda\gtrsim a$, as desired.
\end{proof}

\section{Proofs of Theorems \ref{main eff} and \ref{main prop} on Surfaces}
As we are in the two-dimensional setting, isothermal coordinates provide a convenient way to extend the volume bounds for the superlevel sets of the frequency and the propagation of smallness to the Laplace-Beltrami operator on a closed Riemannian surface with at least Lipschitz regularity. On any Riemannian surface, there exist local coordinates called isothermal coordinates\footnote{See \cite{Chern} for more about isothermal coordinates} in which the Laplace-Beltrami operator takes the form of a scalar multiple of the Euclidean Laplacian. In our propagation of smallness, the constant $C$ obtained will also depend on the ellipticity and Lipschitz constants of the surface's metric, but we suppress this dependence from the notation in order to be more concise. The idea is simple; cover the surface with small coordinate charts such that the map changing to isothermal coordinates has bi-Lipschitz constant close to 1. Consequently, the frequency in the new coordinates does not change much and we can apply the result from the harmonic case on each small ball. The number of charts needed is bounded in terms of the ellipticity and Lipschitz constants of the manifold. This leads to the desired local volume bound for the superlevel sets of the frequency. The propagation of smallness for gradients can be deduced by using the fact that the sizes of the sublevel sets of the gradient of the solution aren't distorted too much when passing to the isothermal coordinates, and a chaining argument can be done to propagate smallness to the entire surface. We can also obtain local bounds on the size of the superlevel sets of the frequency using similar techniques.

In local coordinates, we recall that the Laplace-Beltrami operator for a metric $g$ is a divergence-form elliptic operator given by
\begin{equation}\label{metric for op}
\Delta_gu=\frac{1}{\sqrt{\det g}}\partial_i(\sqrt{\det g}g^{ij}\partial_ju),
\end{equation}
where we are using the Einstein summation convention and $g^{ij}$ denotes the $(i,j)$th component of the inverse of the metric tensor. The key tool is the following result, which controls distortion of the geometry in all directions under the coordinate change, provided the metric is Lipschitz close to being Euclidean.
\begin{prop}\label{deriv bd quasiconf}
Let $g$ be a Lipschitz Riemannian metric on $B_{\eta}(0)\subset \BR^2$ with Lipschitz constant controlled by $\lambda$ and $g(0)=I$ is the Euclidean metric. Assume $\eta$ is sufficiently small (depending on $\lambda$) and let $F:B_{\eta}(0)\tilde\rightarrow U_{\eta}\subset \BR^2$ be the map given by the isothermal change of coordinates. Then the eigenvalues of the differential $DF$ lie in the interval $[1/2,2]$.
\end{prop}
\begin{proof}
The proof will rely on the fact that the existence of isothermal coordinates on surfaces can be rephrased in terms of solvability of the Beltrami equation. We will use many results from \cite{Astala} in this proof, and we will frequently view $\BR^2$ as $\BC$, as in previous sections of the paper. If the metric $g$ takes the matrix form
\[
g=\begin{bmatrix} E & F \\ F & G\end{bmatrix},
\] 
then for
\begin{equation}
    \mu=\frac{E-G+2iF}{E+G+2\sqrt{EG-F^2}},
\end{equation}
the existence of isothermal coordinates is equivalent to finding a homeomorphic solution to
\begin{equation}
    \partial_{\bar{z}}\omega=\mu\partial_z\omega.
\end{equation}
Here, we are extending the definition of $\mu$ to all of $\BC$ by taking a Lipschitz extension with the same Lipschitz constant and $L^{\infty}$ norm and $\mu=0$ outside $B_{2\eta}(0)$ (note that we can do both of these simultaneously since $\mu(0)=0$, such as by interpolating linearly between the values on $\partial B_{\eta}(0)$ and 0 on radial segments). Thus, we have that
\begin{equation}
    |\tilde{\mu}(x)|\le C(\lambda)\eta \chi_{B_{2\eta(0)}},
\end{equation}
where $\chi_E$ denotes the indicator function of the set $E$. The result \cite[Equation 5.25]{Astala} on distortion of Lebesgue measure under quasiconformal mappings with $p=2$ says that
\begin{equation}
    \mathcal{H}^2(B_{2\eta(0)})\le C(\lambda) \eta^2.
\end{equation}
In particular, notice that by ellipticity and the definition of $\mu$, it follows that the Lipschitz constant of $\mu$ depends only on $\lambda$. It is discussed in Chapter 5 of \cite{Astala} (see specifically Theorem 5.1.1, equation (5.4) and the proof of Theorem 5.2.3) that when $|\mu|<1$, then such a solution is given by
\begin{equation}
    \omega(z)=z+\mathcal{C}[(I-\mu\mathcal{S})^{-1}\partial_z\mu]:=z+\sigma(z),
\end{equation}
where $\mathcal{C}$ denotes the Cauchy transform on $\BC$ and $\mathcal{S}$ denotes the Beurling transform. We do not define these here, as we will only be using their $L^p$ boundedness properties as a black box, but a more complete exposition can be found in \cite[Chapter 4]{Astala}. One can compute that the derivative of this mapping (with respect to $z,\bar{z}$) is given by 
\begin{equation}
    D\omega(z)=\begin{bmatrix}
    e^{\sigma} & \mu e^{\sigma} \\
    \overline{\mu e^{\sigma}} & \overline{e^{\sigma}}
    \end{bmatrix}.
\end{equation}
The eigenvalues of this mapping are $\Real(e^{\sigma})\pm (\Real(e^{\sigma})^2+(|\mu|^2-1)|e^{2\sigma}|)^{1/2}$. When $\eta=c(\lambda)$ is sufficiently small, then we have $|\mu|<c'<<1$ is very small (due to the metric being Lipschitz close to the Euclidean metric on the ball), and we will first show that $|\sigma|$ is also small when this happens.

To this end, notice that $|\partial_z\mu(z)|\le C(\lambda)$ holds almost everywhere by Rademacher's theorem and that $\partial_z\mu$ is supported in $B_{2\eta}(0)$. Thus, for $1<p<\infty$, we have that
\begin{equation}
    \Vert \partial_z \mu\Vert_{L^p}\le C(\lambda) \eta^{2/p}.
\end{equation}
The Beurling transform is bounded on $L^p$ for $1<p<\infty$ by the theory of Calder\'on-Zygmund operators (see \cite[Theorem 4.5.3]{Astala}), and its resolvent $(I-\mu \mathcal{S})^{-1}$ is also bounded on $L^p$ with operator norm controlled by $C_p'(1-|\mu|)^{-1}$ (see \cite[Chapter 5]{Astala}) which we can bound in terms of $\lambda$ when $\eta$ is small enough. Thus, we have that
\begin{equation}
    \Vert(I-\mu\mathcal{S})^{-1}\partial_z\mu\Vert_{L^p}\le C_p'\eta^{2/p}C(\lambda)
\end{equation}
holds for all $1<p<\infty$. However, it is shown in \cite[Theorem 4.3.11]{Astala} that the Cauchy transform maps $L^p(\BC)\cap L^q(\BC)\rightarrow C_0(\BC)$ for conjugate pairs $1<q<2<p<\infty$ with the estimate
\begin{equation}
    \Vert \mathcal{C}\phi\Vert_{L^{\infty}}\le \frac{2}{\sqrt{2-q}}\sqrt{\Vert\phi\Vert_{L^p}\Vert\phi\Vert_{L^q}}.
\end{equation}
Taking $\phi=(I-\mu\mathcal{S})^{-1}\partial_z\mu$, we deduce the estimate
\begin{equation}\label{sigma small}
    \Vert \sigma\Vert_{L^{\infty}}\le C_{p,q}'C(\lambda)\eta
\end{equation}
We can fix $q=3/2$ and take $\eta=(10C_{p,q}'C(\lambda))^{-1}$ to conclude that $|\sigma|$ is sufficiently small.

With this established, we return to getting a lower bound on the modulus of the eigenvalues. Recall throughout that we have upper and lower bounds on $|e^{\sigma}|$ by the above argument. We will consider cases based on if the eigenvalues are real or complex and on whether $e^{\sigma}$ has a large or small real part.

\textbf{Case 1}: Suppose that the eigenvalues are complex, i.e. that
\[
\Real(e^{\sigma})^2+(|\mu|^2-1)|e^{2\sigma}|<0.
\]
If $|\Real(e^{\sigma})|^2>c_{\lambda}|e^{2\sigma}|$ for some constant $c_{\lambda}$ depending only on $\lambda$, then we are done since we have a lower bound on the real part of the eigenvalue and hence also on its modulus. Thus, we may assume the inequality fails for some small constant $c_\lambda$. Taking it smaller than $(1-|\mu|^2)/2$ (where $|\mu|$ depends on $\lambda$), we see that we can lower bound the imaginary part of the eigenvalue, completing this case.

\textbf{Case 2}: Suppose that the eigenvalues are real, i.e. that
\[
\Real(e^{\sigma})^2+(|\mu|^2-1)|e^{2\sigma}|\ge 0.
\]
In the case where the discriminant is 0, we can write $\Real(e^{\sigma})$ as a multiple of $|e^{\sigma}|$ and get the desired lower bound, so we can assume the discriminant is positive. In general, it follows by \eqref{sigma small} that
\[
\Real(e^{\sigma})^2\ge e^{2(1-c\eta)},
\]
so we just need to show that the summand arising from the discriminant is not too close to this quantity. If the discriminant is less than $c_{\lambda}\Real(e^{\sigma})^2$ for $c_{\lambda}<1$ then we are done and get a lower bound of $(1-c_{\lambda})\Real(e^{\sigma})^2$. Fix $c_{\lambda}=|\mu|^2$. Then if the previous discriminant estimate fails, we can derive that
\[
\Real(e^{\sigma})^2>\frac{1-|\mu|^2}{1-c_\lambda}|e^{2\sigma}|=|e^{2\sigma}|
\]
which is a contradiction. Thus, we always have the desired lower bound on the modulus of the eigenvalues.
\end{proof}
With this established, the propagation of smallness follows in a straightforward fashion.
\begin{proof}[Proof of Theorem \ref{main prop}]
This is already known in the case where the operator is the Euclidean Laplacian, so it suffices to reduce to that case. Around each point $x$ in $B_{1/2}(0)$, consider an ellipsoidal disk (that is, the set of points on and inside an ellipse) with major axis having length at most $\eta(2\lambda)^{-1}$, where $\eta$ is as in the previous lemma, and the ellipse is chosen so that the linear change of coordinates taking $g(x)$ to the identity matrix sends the corresponding boundary ellipse to the circle of radius $\eta/2$. By subadditivity of the Hausdorff content, we can select one of these ellipsoidal disks $\Theta_0$ with the property that $\mathcal{C}_{H}^{\delta}(\Theta_0\cap E_{\epsilon})\ge c(\lambda)\beta$. We can pick $C(\lambda)$ many points on the boundary of $\Theta_0$ and take the corresponding ellipsoidal disks, $\Theta_{1,j}$s, so that $(1+c(\lambda))\Theta_0$ is covered by $\Theta_0$ and the $\Theta_{1,j}$s (this is a consequence of ellipticity). We can then iterate this procedure until $B_{1/2}(0)$ is covered by $C(\lambda)$ many ellipsoidal disks. Moreover, in this construction, we can ensure by ellipticity that we have a lower bound on the fraction of each ellipsoidal disk that overlaps with the previous disks, call it $b(\lambda)$. On $\Theta_0$, we can consider its double and make a linear change of coordinates so that the double becomes a ball of radius at most $\eta$. Let $\tilde{u}$ denote the harmonic function obtained by composing with the isothermal change of coordinates $F$. This satisfies
\[
|\nabla \tilde{u}(F^{-1}(x))|\le 2|\nabla u(x)|
\]
so it follows that $E_{a-\log 2}(\tilde{u})\supset F(E_{a}(u))$. Control of the eigenvalues of $DF$ tells us that we still have a lower bound
\[
\mathcal{H}^{\delta}(E_{a-\log 2}(\tilde{u}))\ge C \beta;
\]
hence, by Theorem \ref{thm: harm prop}, we have that
\[
|\nabla \tilde{u}(x)| \le Ce^{\gamma(-a+\log 2)}
\]
for all $x$ in $B_{\eta/2}$. We will absorb the $\log 2$ into the multiplicative constant $C$. Now, $F^{-1}(B_{\eta/2})$ covers $c\Theta_0$ for some universal $c<1$. Thus, we have that 
\begin{equation}\label{first prop}
|\nabla u(x)|\le C'e^{-\gamma a}
\end{equation}
for all $x\in c\Theta_0$. We can chain this argument to get that the same inequality holds throughout $\Theta_0$ with worse values of $C',\gamma$. 

Now, consider the doubles of each of the $\Theta_{1,j}$s. We have that for all $x$ in a set with $\delta$-Hausdorff content bounded below by $C(\lambda)$ that \eqref{first prop} holds. This as a result of the fact that we ensured that there is a lower bound on the overlap between the various ellipsoidal disks. We can propagate smallness using the same argument as for $\Theta_0$, getting more factors in the constants depending on $\lambda$. Chaining this argument iteratively completes the proof, since the number of steps required to cover the original disk of radius $1/2$ is also controlled in terms of $\lambda$. 
\end{proof}

We can also extend the volume estimates on the superlevel sets of the frequency to the case of Lipschitz manifolds.

\begin{proof}[Proof of Theorem \ref{main eff}]
As in the previous proof of Theorem \ref{main prop}, we can cover the half ball with ellipsoidal disks with major axis having length at most $\eta(2\lambda)^{-1}$, such that the linear change of coordinates taking the metric tensor to the identity at the center of each ellipse sends the corresponding boundary ellipse to the circle of radius $\eta/2$. By almost monotonicity of the frequency function, the frequency on these smaller domains is at most $C\Lambda$. Applying the quasiconformal map, we get that our function $\tilde{u}$ in the new coordinates is harmonic and has frequency at most $C\Lambda$. The harmonic case of Theorem \ref{main eff} implies
\begin{equation}
    \vol(\{z:N_{\tilde{u}}(z,r)>1\}\cap B_{\eta/2}(0))\le C\Lambda^2 r^2
\end{equation}
Applying the inverse of the quasiconformal map distorts this volume by a bounded constant factor, and by choosing the constant $c$ in the statement of Theorem \ref{main eff} sufficiently small, we have that its image covers $\{z:N(z,r)>c\}$ in the ellipsoidal disk. Repeating this argument for each such ellipsoidal disk and summing the volume contributions completes the proof.
\end{proof}

\section{More General Elliptic Equations with Lipschitz Coefficients}\label{section gen eqn}
In this section, we extend the estimate for the size of the superlevel sets of the frequency to more general elliptic equations at the cost of a slightly worse bound. Consider the operator
\begin{equation}
    Lu:=\Div(A\nabla u)+\tilde{\mathbf{b}}\cdot \nabla u,
\end{equation}
where $A$ is uniformly elliptic with constant $\lambda$ and has Lipschitz entries with constant controlled by $\lambda$ as well. We assume $\tilde{\mathbf{b}}$ has components that are bounded and measurable with size controlled by $\lambda$. The key is that for solutions to $Lu=0$, the logarithm of the modulus of the gradient will solve an elliptic equation, and we will make use of several of the results developed in \cite{Ale1}, \cite{Ale2}, and \cite{Zhu}. As we are in two dimensions, we may work locally in isothermal coordinates as before, where the equation takes the form
\begin{equation}\label{euc lap with drift}
\Delta u + \mathbf{b}\cdot\nabla u=0,
\end{equation}
on some domain which we will without loss of generality take to be a small Euclidean ball $B_{\kappa}(0)$, and $\mathbf{b}$ still has bounded, measurable components. Following the work of \cite{Ale1}, if we set $\phi=\log|\nabla u|^2$ then $\phi$ is a weak solution to the following equation away from the critical set:
\begin{align}\label{PDE w RHS}
    \Delta \phi = -\Div\left(\frac{\mathbf{b}\cdot\nabla u}{|\nabla u|^2}\nabla u\right)
\end{align}
For the reader's convenience, we present the calculation in this simplified case of \eqref{euc lap with drift} where the operator is the Euclidean operator with a bounded drift term in the appendix.
\begin{prop}\label{prop: log grad equation}
    Suppose $u$ solves \eqref{euc lap with drift}. Then on $B_{\kappa}(0)\setminus\mathcal{C}(u)$, we have that $\phi=\log|\nabla u|^2$ is a weak solution to \eqref{PDE w RHS}
\end{prop}

\begin{proof}
    This is deferred to the appendix.
\end{proof}

As in the harmonic case, we want to factorize the gradient into a nonvanishing component and a polynomial component which can be analyzed separately. Denote the points of the critical set by $z_1,...,z_N$, repeating as needed when there is multiplicity. By the work of \cite{Zhu}, we have $N\le C\Lambda$. We write
\[
e^{\phi(z)}=|\nabla u(z)|^2=\prod_{j=1}^N|z-z_j|^2 e^{\psi(z)}:=P(z) e^{\psi(z)}
\]
where $\psi$ is a bounded function and we let $P$ denote the polynomial in the factorization. Now, we reformulate a lower bound for our setting.
\begin{prop}[Remark 3 in \cite{Ale2}]
Under the hypotheses of this section, $|\psi(z)|\le C\Lambda$ for some universal $C>0$ and all $z$ in the unit disk.
\end{prop}
By Lemmas 2.2 and 2.3 in \cite{Ale1}, we have that $\psi$ solves the PDE \eqref{PDE w RHS} in the entire domain, not just away from the critical set. By Calder\'on-Zygmund theory, we have the bound
\begin{equation}\label{BMO bound for derivative}
\Vert\nabla\psi\Vert_{BMO(B_{\kappa/2}(0))}\le C (\Vert\psi\Vert_{L^{\infty}(B_\kappa(0))}+\Vert\mathbf{b}\Vert_{L^{\infty}(B_\kappa(0))}).
\end{equation}
Indeed, this follows from the representation formula
\begin{align*}
\psi(x)&=\int_{\partial B_{\kappa}(0)}K_{\kappa}(x,y)\psi(y)\,dS(y)+\int_{B_{\kappa}(0)}\nabla_yG(x,y)\cdot\nabla u(y)\frac{\mathbf{b}(y)\cdot\nabla u(y)}{|\nabla u(y)|^2}\,dy\\
&-\int_{\partial B_{\kappa(0)}}G(x,y)\frac{\mathbf{b}(y)\cdot\nabla u(y)}{|\nabla u(y)|^2}\partial_{\nu}u(y)\,dS(y)=(I)+(II)+(III)
\end{align*}
where $K_{\kappa}$ is the Poisson kernel for the ball of radius $\kappa$ and $G(x,y)=\log|x-y|$ is the Green's function. Taking the gradient in $x$, the $(I)$ and $(III)$ terms are bounded for $x\in B_{\kappa/2}(0)$, and the term $(II)$ can be recognized as a vector-valued Calder\'on-Zygmund operator, which is known to map $L^{\infty}$ to $BMO$ continuously.
Under our assumptions, this in particular implies
\begin{equation}
    \Vert \nabla \psi\Vert_{BMO(B_{\kappa/2}(0))}\le C\Lambda +C',
\end{equation}
and we can absorb the $C'$ by assuming $\Lambda$ is sufficiently large. With this substitute for Lipschitz control established, we can now show that the frequency of $e^{\psi}$ is very small at scales below $1/\Lambda^{1+\epsilon}$.
\begin{prop}\label{prop: weaker nonvanishing}
    Let $e^{\psi}$ be defined as above and assume $\Lambda$ is sufficiently large. For every $\epsilon>0$, there is a constant $c_{\epsilon}$ such that if $r<c_{\epsilon}/\Lambda^{1+\epsilon}$, then
    \[
    N_{\exp(\psi)}(x,r)\le \frac{3}{4}
    \]
    holds uniformly in $x$.
\end{prop}
\begin{proof}
We consider the $L^2$ doubling index, i.e.
    \[
    N(x,r)=\log_2\left(\frac{\dashint_{B_{2r}(x)} \exp(-2\psi(y))\,dy}{\dashint_{B_{r}(x)}\exp(-2\psi(y))\,dy}\right):=\log_2\left(\dashint_{B_{2r}(x)} \exp(-2\psi(y)+2\psi(y_0))\right)
    \]
    where $y_0$ is a point in $B_r(x)$ where the average in the denominator is achieved. Define $g_{y_0}(y)=\int_0^1 \nabla\psi((1-t)y_0+ty)$ so that $g_{y_0}(y)\cdot(y-y_0)=\psi(y)-\psi(y_0)$ by the fundamental theorem of calculus. By scale invariance of BMO, we have that 
\begin{equation}\label{BMO grad easy estimate}
\Vert g_{y_0}\Vert_{BMO}\le C\Vert \nabla \psi\Vert_{BMO}.
\end{equation}
We will make use of the notation $\langle f \rangle_U$ to denote the average of a function (possibly vector-valued) over an open set $U$. Thus, we rewrite
    \begin{equation}\label{exp split}
    N(x,r)=\log_2\left( \dashint_{B_{2r}(x)} e^{[g_{y_0}(y)-\langle g_{y_0}\rangle_{B_{2r}(x)}]\cdot(y-y_0)}e^{\langle g_{y_0}\rangle_{B_{2r}(x)}\cdot (y-y_0)} \right).
    \end{equation}
    We will start by bounding the $L^{\infty}$ norm of the second factor. Recall that $\Vert \nabla\psi\Vert_{L^p(B_1(0))}\le C_p \Lambda$ for all $1\le p<\infty$. Thus, we have that
    \[
    \frac{1}{Cr^2}\int_{B_r(x)}|\nabla \psi(y)|^2\,dy\le c\frac{1}{r^2}\Vert \nabla \psi\Vert_{L^p} r^{2/p'}\le C_p \Lambda r^{2(1/p'-1)}.
    \]
    We can choose $p=p(\epsilon)$ sufficiently large to make the exponent less than any $\epsilon>0$. Thus, we have that
    \[
    |\langle g_{y_0}\rangle_{B_{2r}(x)}\cdot (y-y_0)|\le C_{\epsilon}\Lambda r^{1-\epsilon}.
    \]
    Now, we estimate the $L^1$ norm of the first factor in \eqref{exp split}. We make use of the following corollary of the John-Nirenberg inequality
    \[
    \frac{1}{|B|}\int_B\exp(\gamma |f(x)-\langle f\rangle_B|/\Vert f\Vert_{BMO})\, dx \le 1+\frac{4e^2\gamma}{1-4e\gamma},
    \]
    valid if $\gamma<(4e)^{-1}$, see e.g. \cite[Corollary 3.1.7]{Grafakos2} which presents this for averages over cubes (the formulation for balls is proved analogously). In our setting, this gives
    \[
    \dashint_{B_{2r}(x)} \exp(|g_{y_0}-\langle g_{y_0}\rangle_{B_{2r}(x)}|\cdot|y-y_0|)\,dy \le 1+\frac{8e^2 r\Vert g_{y_0}\Vert_{BMO}}{1-8re\Vert g_{y_0}\Vert_{BMO}}.
    \]
    We can combine this with the estimate \eqref{BMO grad easy estimate} and the $L^{\infty}$ normalization of $\nabla\psi$. Thus, if we choose $r<c_{\epsilon}\Lambda^{-1-\epsilon}$, we see that $N(x,r)\le 1/2+1/\Lambda$, as desired.

\end{proof}

For the polynomial part, we can reuse Proposition \ref{prop: poly vol}. We can now estimate the size of the superlevel sets of the frequency. The proof follows by repeating the argument from the harmonic case using the slightly worse bound for the nonvanishing part given in Proposition \ref{prop: weaker nonvanishing}. Thus, for every $\epsilon>0$, we get the following nearly quadratic bound in a single isothermal coordinate patch:
\[
\vol(\{z:N(z,r)>c\}\cap B_{\kappa/2}(0))\le C_{\epsilon}\Lambda^{2+\epsilon} r^2.
\]
Covering the interior of our original domain with overlapping isothermal coordinate patches allows us to extend the bound to the interior of the original domain. Thus, we conclude the following generalization.
\begin{thm}\label{eff crit general eqn}
    Suppose $u$ is a solution to $\Div(A\nabla u)+\mathbf{b}\cdot\nabla u=0$ on the Euclidean disk of radius 4. Assume its frequency satisfies $N_u(0,1)\le \Lambda$. Then for every $\epsilon>0$, there is a constant $C_{\epsilon}$ such that we have the following volume estimate
\begin{equation}
    \vol(\{z:N(z,r)>c\}\cap B_{1/2}(0))\le C_{\epsilon}\Lambda^{2+\epsilon} r^2,
\end{equation}
where $C_{\epsilon}$ depends only on the Lipschitz constants of $A,\mathbf{b}$ and the ellipticity constant of $A$.
\end{thm}

\section*{Appendix: Proof of Proposition \ref{prop: log grad equation}}
For convenience, we demonstrate this in the case of additional regularity in which case we verify it is a strong solution, but the argument is analogous when verifying the result for weak solutions. Let $v=|\nabla u|^2$. We start by computing
\begin{align*}
\Delta (\log v)&=\frac{1}{v^2}\Div(2 D^2u\nabla u)+2D^2u\nabla u\cdot\nabla(v^{-1}) \\
&=\frac{2}{v}(|D^2u|^2+\nabla u\cdot\nabla(\Delta u))-4D^2u\nabla u\cdot\frac{D^2u\nabla u}{v^2}
\end{align*}
Our goal is to show that this equals $-\Div\left(\frac{\mathbf{b}\cdot\nabla u}{|\nabla u|^2}\right)$, which we can rewrite as
\begin{align*}
\Div\left(\frac{\Delta u}{v^2}\right)&=2\nabla\left(\frac{\Delta u}{v}\right)\cdot\nabla u+2\frac{(\Delta u)^2}{v} \\
&=\frac{2}{v}\nabla(\Delta u)\cdot\nabla u-4\Delta u\frac{(D^2 u\nabla u)\cdot\nabla u}{v^2}+2\frac{(\Delta u)^2}{v}
\end{align*}
We notice the common term of $\frac{2}{v}\nabla u\cdot\nabla(\Delta u)$, so it is enough to verify the remaining terms agree. It will be convenient to consider the relevant quantities multiplied by $v^2$. At this point, the equality will depend on the fact that we are in two dimensions, so we expand in terms of partial derivatives, which we'll denote by $u_x,u_y,u_{xx},u_{xy},u_{yy}$. The equality we want to verify is
\begin{equation}\label{goal eqn}
2v|D^2u|^2-4|D^2u\nabla u|^2=-4\Delta u(D^2u\nabla u)\cdot\nabla u+2v(\Delta u)^2
\end{equation}
The left side of \eqref{goal eqn} can be expanded as
\[
2(u_x^2+u_y^2)(u_{xx}^2+2u_{xy}^2+u_{yy}^2)-4(u_{xx}u_x+u_{xy}u_y)^2-4(u_{xy}u_x+u_{yy}u_y)^2=(I)+(II)+(III)
\]
The right side of \eqref{goal eqn} can be expanded as
\[
-4(u_{xx}+u_{yy})(u_{xx}u_x^2+2u_{xy}u_xu_y+u_{yy}u_y^2)+2(u_x^2+u_y^2)(u_{xx}+u_{yy})^2=(IV)+(V)
\]
Notice that $(V)-(I)=4(u_x^2+u_y^2)(u_{xx}u_{yy}-u_{xy}^2)$. We can consolidate
\begin{align*}
(II)+(III)&=-4(u_{xx}^2u_x^2+2u_{xx}u_{xy}u_xu_y+u_{xy}^2u_y^2+u_{xy}^2u_x^2+2u_{yy}u_{xy}u_xu_y+u_{yy}^2u_y^2) \\
&= -4(u_{xx}^2u_x^2+2u_{xx}u_{xy}u_xu_y+2u_{yy}u_{xy}u_xu_y+u_{yy}^2u_y^2)-4(u_x^2+u_y^2)u_{xy}^2
\end{align*}
On the other hand, we have that
\begin{align*}
(IV)+((V)-(I))=-&4(u_{xx}^2u_x^2+2u_{xx}u_{xy}u_xu_y+u_{xx}u_{yy}u_y^2 \\
&+u_{xx}u_{yy}u_x^2+2u_{yy}u_{xy}u_xu_y+u_{yy}^2u_y^2)\\
&+4(u_x^2+u_y^2)u_{xx}u_{yy}-4(u_x^2+u_y^2)u_{xy}^2
\end{align*}
At this point, it is very straightforward to verify that the terms in $(II)+(III)$ agree with those in $(IV)+(V)-(I)$, completing the proof.
\bibliographystyle{alpha}
\bibliography{sources}

\end{document}